\documentclass[11pt]{amsart}

\usepackage{amscd,amssymb,amsopn,amsmath,amsthm,mathrsfs,graphics,amsfonts,enumerate,verbatim,calc
}
\usepackage{bbm}
\usepackage[all,cmtip]{xy}
\usepackage[all]{xy}
\usepackage{tikz}
\usepackage{lscape}
\usepackage{stmaryrd}
\usetikzlibrary{matrix,arrows,decorations.pathmorphing}

\usepackage{scalerel}
\usepackage{stackengine,wasysym}
\usetikzlibrary {positioning}
\usetikzlibrary{patterns}
\usetikzlibrary{calc}
\definecolor {processblue}{cmyk}{0.96,0,0,0}

\usepackage{subfigure}

\usepackage{comment} 

\usepackage{ dsfont }

\usepackage{color}

\usepackage[OT2,OT1]{fontenc}
\newcommand\cyr{%
\renewcommand\rmdefault{wncyr}%
\renewcommand\sfdefault{wncyss}%
\renewcommand\encodingdefault{OT2}%
\normalfont
\selectfont}
\DeclareTextFontCommand{\textcyr}{\cyr}

\usepackage{amssymb,amsmath}

\DeclareFontFamily{OT1}{rsfs}{}
\DeclareFontShape{OT1}{rsfs}{n}{it}{<-> rsfs10}{}
\DeclareMathAlphabet{\mathscr}{OT1}{rsfs}{n}{it}

\numberwithin{equation}{section}
\newcounter{thm}
\newtheorem{theorem}{Theorem}[section]

\newtheorem{lem}[theorem]{Lemma}
\newtheorem{cor}[theorem]{Corollary}

\theoremstyle{definition}
\newtheorem{prop}[theorem]{Proposition}

\theoremstyle{remark}
\newtheorem{remark}[theorem]{Remark}

\newtheorem{example}[theorem]{Example}

\newcommand{\Q}{\mathbb{Q}}
\newcommand{\N}{\mathbb{N}}

\newcommand{\Ass}{\operatorname{Ass}}

\newcommand{\im}{\operatorname{im}}

\newcommand{\Spec}{\operatorname{Spec}}

\newcommand{\pd}{\operatorname{pd}}
\newcommand{\id}{\operatorname{id}}
\newcommand{\Ext}{\operatorname{Ext}}
\newcommand{\Supp}{\operatorname{Supp}}

\newcommand{\Hom}{\operatorname{Hom}}

\newcommand{\depth}{\operatorname{depth}}

\newcommand{\coker}{\operatorname{coker}}

\newcommand{\rank}{\operatorname{rank}}

\newcommand{\p}{\mathfrak{p}}
\newcommand{\m}{\mathfrak{m}}
\newcommand{\q}{\mathfrak{q}}
\newcommand{\w}{\omega}
\DeclareMathOperator{\DF}{DF}
\DeclareMathOperator{\CM}{CM}
\DeclareMathOperator{\Deep}{Deep}
\DeclareMathOperator{\End}{End}
\DeclareMathOperator{\hit}{ht}

\DeclareMathOperator{\Ann}{Ann}
\DeclareMathOperator{\Soc}{Soc}
\DeclareMathOperator{\Min}{Min}
\DeclareMathOperator{\Tr}{Tr}
\renewcommand{\mod}{\ensuremath{\operatorname{mod}}}

\begin{document}
\title[Hom and Ext, Revisited]{Hom and Ext, Revisited}

\author[Dao]{Hailong Dao}
\email[Hailong Dao]{hdao@ku.edu}

\address{Department of Mathematics\\
University of Kansas\\
Lawrence, KS 66045-7523 USA}

\author[Eghbali]{Mohammad Eghbali}
\email[Mohammad Eghbali]{m.eghbali007@yahoo.com}

\address{Faculty of Mathematical Sciences and Computer \\
Kharazmi University \\  Tehran,  Iran}

\author[Lyle]{Justin Lyle}
\email[Justin Lyle]{justin.lyle@ku.edu}

\address{Department of Mathematics\\
University of Kansas\\
Lawrence, KS 66045-7523 USA}

\date{\today}


\begin{abstract}
Let $R$ be a commutative Noetherian local ring and $M,N$ be finitely generated $R$-modules. We prove a number of results of the form: if $\Hom_R(M,N)$ has some nice properties and $\Ext^{1 \leq i \leq n}_R(M,N)=0$ for some $n$, then $M$ (and sometimes $N$) must be  close to free. Our methods are quite elementary, yet they suffice to give a unified treatment, simplify, and sometimes extend a number of results in the literature. 

\end{abstract}

\dedicatory{Dedicated to Professor Craig Huneke on the occasion
     of his $65$th  birthday}

\maketitle

\section{Introduction}

Let $R$ be a commutative Noetherian local ring and $M,N$ be finitely generated $R$ modules. The original purpose of this project is to understand a large and growing body of results which take the form: if $\Hom(M,N)$ has some nice properties and $\Ext^{1 \leq i \leq n}_R(M,N)=0$ for some $n$, then $M$ and $N$ must be nice themselves. 

For example, about 50 years ago Vasconcelos proved that if $R$ is a Gorenstein ring of dimension $1$, and $M$ is a maximal Cohen-Macaulay (MCM) $R$-module such that $\End_R(M)$ is free, then $M$ is free (\cite{VA68}). Ulrich proposed tests for the Gorensteiness of $R$ using $\Ext$-vanishing between certain modules and $R$ (\cite{UL84}).  Huneke and Leuschke proved an interesting special case of the famous Auslander-Reiten conjecture. One of the main results says that if $R$ is a normal domain of dimension $d$ and $M$ is a module locally free in codimension one, and if $\Ext^{1\leq d}_R(M,M)= \Ext^{1\leq j\leq 2d+1}_R(M,R)=0$, then $M$ must be free (\cite{HL04}). These influential results have been examined and extended quite frequently, see \cite{AC17, CI16, GT17, LI17a, LI17b} for a sample of some interesting work appearing just within the last year and the references therein. These papers all serve as the main inspiration for our work.

Our approach to the questions above is to first study, as thoroughly as we can,  the small dimension or depth situation. This is important in our view since most of the proofs involve an inductive process by localization or cutting down with a regular sequence. Surprisingly, this simple-minded approach makes the problems more transparent and yields significant improvements; we can usually  remove assumptions such as Cohen-Macaulayness, constant rank, $M=N$, etc.,  altogether. At the same time,  proofs become shorter and more elementary. In fact, we do not need much preparatory material beyond graduate level commutative algebra. Thus we hope that our paper will be useful for young people just starting on the subjects.

We now describe our work in more detail. Let $R$ be a local ring of dimension $d$ and depth $t$. In Section \ref{deep} we define two categories of modules that are crucial for our analysis. One is called $\Omega\Deep(R)$, which consists of modules $M$ that are a syzygy of some high-depth module. That is, such an $M$ fits into an exact sequence $0\to M\to F\to X\to 0$ with $F$ free and $\depth X\geq t$. Somewhat dually, the second category, $\DF(R)$, consists of $M$ such that there is an exact sequence $0\to F\to M^n \to X\to 0$ with $F$ free and $\depth X\geq t$ ($\DF$ stands for ``deeply faithful"). We establish a number of simple but useful results about these categories. For example, they behave well with respect to ``cutting down by a general regular sequence", and  any object lying in their intersection must have a free summand (we actually prove a bit more, see Theorem \ref{corintersectfreesummand}). 

In Section \ref{MNfree} we study the question: when does $\Hom_R(M,N)$ have a free summand or is free? Here our first main result is:

\setcounter{section}{3}
\setcounter{theorem}{5}

\begin{theorem}\label{intro1}
Suppose that $\depth M\geq t$ and $N\in \Omega \Deep(R)$. Assume that $\Hom _R(M,N) \in \DF(R)$ and  $\Ext^{1 \leqslant i \leqslant t-1}_R(M,N)=0$. Then $N$ has a free summand. 
\end{theorem}

This allows us to generalize both the results by Vasconcelos and Huneke-Leuschke mentioned above in Corollary \ref{MNFree} and Theorem \ref{HunekeLeuschkegen}. We also prove that if $R$ and $M$ satisfy Serre's condition $(S_2)$ and $\Hom_R(M,R)$ is free, then $M$ is free (see \ref{generalMdual}).  The key point here is the dimension one case. Lastly, we  extend a result by Goto-Tatakashi to higher rank modules (see Theorem \ref{GTgen}).

In Section \ref{MNN} we study when $\Hom_R(M,N)\cong N^r$ for some $r>0$. Again we start with the small depth or dimension situation and build from there. Our main technical result is:
\setcounter{section}{4}
\setcounter{theorem}{6}

\begin{theorem}\label{intro2}
Assume that   $\depth(N)=t$, $\depth(M) \geqslant t$, $\Ass(N)=\Min(N)$,  and for some $s\geq t$,  $\Ext^{1 \leqslant i \leqslant s} _R (M, N)=0$. If $\Hom(M, N) \cong N^{r}$ for some $r \in \N$, then $M/IM \cong (R/I)^r$ for $I=\Ann(N)$.

Furthermore, if  one of the following holds:
\begin{enumerate}
\item $N$ is faithful. 
\item $\Ass(R) \subseteq \Ass(N)$ and $s>0$.  
\end{enumerate}  
then $M \cong R^{r}$.
\end{theorem}

\setcounter{section}{1}
\setcounter{theorem}{1}
We give some applications, including a modest case of the Auslander-Reiten conjecture (Corollary \ref{MM}).

In the last section, we address a couple of related topics: a test for Gorensteiness inspired by an old result of Ulrich (Corollary \ref{testgor}), and an equivalent condition for vanishing of $\Ext$ modules that slightly extends results by Huneke-Hanes, Huneke-Jorgensen and Huneke-Leuschke in \cite{HH05, HJ03, HL04}, see Corollary \ref{tight}. 

\section{Two key categories}\label{deep}

Throughout $(R,\m)$ is a Noetherian local ring with $\dim(R)=d$ and $\depth(R)=t$. In this section we define and establish basic facts about two categories of modules that play a crucial role for many of our proofs. 

But first, some notation. We let $\mu(M)$ denote the minimal number of generators of a module $M$ and $l(M)$ its length. We say that $M$ is generically free if $M_{\p}$ is free over $R_{\p}$ for any $\p \in \Ass(R)$. Let $(S_i)$ denote Serre's condition: $\depth M_{\p}\geq \min\{i, \hit \p\}$.

$M$ is said to be free in codimension $n$ if $M_{\p}$ is free for each prime $\p$ of height at most $n$. $M$ has a rank if it is generically free and the rank over all $\p \in \Ass(R)$ is constant. We use the notation $M \mid N$ to say that $M$ is a summand of $N$. 

For $R$, being  $(G_j)$ means Gorenstein in codimension $j$. Let $e(\underline x,M)$ denote the multiplicity of $M$ with respect to a sequence $\underline x$. Without further comment, we will often use the notation $\overline{(-)}=- \otimes_R R/\underline{x}$ when the sequence $\underline{x}$ is clear from context.  As usual, $\mod(R)$ and $\CM(R)$ denote the category of finitely generated  and maximal Cohen-Macaulay modules respectively. 

We let  $\Deep(R) =\{X  \mid \depth(X) \geqslant t\}$.  If $\mathcal C$ is a category, we use the notation 
\[\Omega\mathcal C:=\{M \mid \exists \ 0 \to M \to R^n \to X \to 0 \mbox{ exact for some}\ X\in \mathcal C\}.\] We let $\Omega^{i+1}\mathcal C:= \Omega\Omega^i\mathcal C$. The first category that is important to us is $\Omega\Deep(R)$. That is:
\[\Omega \Deep (R):= \{M \mid \exists\ 0\rightarrow M\rightarrow R^{n} \rightarrow X\rightarrow 0\ \mbox{exact for some}\ n\in \N\ \mbox{and}\ X \in \Deep(R)\},
\]
We also consider: \[\DF(R):=\{M \mid \mbox{$\exists$ \ $0\rightarrow R\rightarrow M^{n} \rightarrow X\rightarrow 0$ exact for some $n\in \N$ and $X \in \Deep(R)$}\}.\]

We define $\Omega_{\min} \Deep(R)=\{M \mid \exists\ 0 \to M \to R^{\mu(X)} \to X \to 0 \mbox{ exact with}\ X\in \Deep(R)\}$.  If $M \in \Omega_{\min} \Deep(R)$, we say that $M$ is a minimal syzygy in $\Omega \Deep(R)$.

If $M \in \Deep(R)$, we say that $M$ is a deep module, and likewise if $M \in \DF(R)$, we say that $M$ is deeply faithful.  

\begin{remark}
\begin{enumerate}

\item $R$ is in $\Deep(R)\cap \DF(R)$. Any $t$-syzygy module is in $\Deep(R)$ (in other words $\Omega^t\mod(R) \subseteq \Deep(R)$).
\item Both $\Omega\Deep(R)$ and $\DF(R)$ are subcategories of $\Deep(R)$.  
\item It is clear that for any $X \in \DF(R)$, $X$ is faithful, and when $\depth(R)=0$ deeply faithful and faithful modules coincide.

\item One can see that
\[ \Omega^{t+1}(\mod R) \subseteq \Omega \Deep(R).\]

\item If $R$ is Cohen-Macaulay, then $\Deep(R)=\CM(R)$. Furthermore, $R$ is Gorenstein if and only if $\Omega\CM(R)=\CM(R)$. 

\item In case $R$ is Cohen-Macaulay and admits a canonical module $\omega_R$ then, for an $R$-module $M$, we set $M^{\vee}=\Hom_R(M,\omega_{R})$. In this case we have 

\[{\Omega \CM(R)}=\{X^\vee \mid X \in \CM(R), \mbox{$\exists$ $\omega_{R}^{n} \twoheadrightarrow X$, for some $n \in \N$}\},\] and
\[\DF(R) = \{X^\vee \mid X\in \CM(R), \mbox{$\exists$ $X^{n}\twoheadrightarrow \omega_{R}$, for some $n \in \N$}\}.\]
\item If $M$ is a {\it semi-dualizing} module (that is, if the natural map $R \to \Hom_R(M,M)$ is an isomorphism and $\Ext^{i>0}_R(M,M)=0$), then $M \in \DF(R)$. See Lemma \ref{summandlem}.
\end{enumerate}
\end{remark}

\begin{remark}\label{general}
When the proof of a statement involves finitely many objects in $\Deep(R)$, one can use prime avoidance to find a regular sequence $\underline x$ of length at most $t$ on all of them. In such situations we shall often say, without further comments, that $\underline x$ is a {\it general} regular sequence. 
\end{remark}

We illustrate the above remark in the following simple but useful result:  
\begin{lem}\label{deepcutdown}
Suppose $t>0$.  If $M \in \Omega \Deep(R)$ (resp. $M \in \DF(R)$) then, for a general regular sequence $\underline{x}$, we have $\overline{M} \in \Omega \Deep(\overline{R})$ (resp. $\overline{M} \in \DF(\overline{R})$).
\end{lem} 

\begin{proof}
As $M \in \Omega \Deep(R)$ (resp. $M \in \DF(R)$), there is an exact sequence $0 \to M \to R^n \to X \to 0$ with $X \in \Deep(R)$ (resp. $0 \to R \to M^n \to X \to 0$ with $X \in \Deep(R)$). Then, as in Remark \ref{general}, for sufficiently general $x \in \m$, $x$ is regular on $R$ and $X$, and so the sequences remain exact modulo $x$.  The result then follows from induction. 
\end{proof}

\begin{lem}\label{minsyz}
Assume that $t=0$, $0 \neq M \in \Omega\Deep(R)= \Omega\mod(R)$. The following are equivalent.
\begin{enumerate}
\item $R \mid M$.    
\item $M$ is faithful.
\item $\Soc(R) \nsubseteq \Ann(M)$.
\item $M$ is not a minimal syzygy on $R$.
\end{enumerate} 
\end{lem}

\begin{proof}
$(4)\Rightarrow (1) \Rightarrow (2) \Rightarrow (3)$ are clear.
For $(3)\Rightarrow (4)$, assume that $M$ is a minimal syzygy. Consider the sequence $0 \to M \to R^n\to X \to 0$ with $n=\mu(X)$. If $M$ is free then $X$ has finite projective dimension, hence free, impossible.  Thus $M$ is not free, so it is a part of an infinite minimal resolution. Therefore we have $M \subseteq \m R^{n}$ for some $n$, so $\Soc(R)M=0$.   

\end{proof}

\begin{lem}\label{pushouts}
Let $0\rightarrow X\xrightarrow{i_1} Y\xrightarrow{p_1} Z\rightarrow 0$ be an exact sequence.
\begin{enumerate}
\item If  $Y \in \Omega \Deep(R)$ and $Z \in \Deep(R)$ then $X \in \Omega \Deep(R)$.

\item If $X \in \DF(R)$ and $Z \in \Deep(R)$ then $Y \in \DF(R)$.
\end{enumerate}
\end{lem}

\begin{proof}

\begin{enumerate}

\item Since $Y \in \Omega \Deep(R)$, there is an exact sequence of the form $0 \rightarrow Y \xrightarrow{i_2} R^n \xrightarrow{p_2} C \rightarrow 0$, where $C \in \Deep(R)$.  Letting $P$ be the pushout along $p_1$ and $i_2$, we have the following pushout diagram with exact rows and columns: 
\[\begin{tikzpicture}[descr/.style={fill=white,inner sep=1.5pt}]
       \matrix (m) [
            matrix of math nodes,
            row sep=2em,
            column sep=2em,
            text height=1.5ex, text depth=0.25ex
        ]
        {\mbox{} & 0 & 0 & 0 & \mbox{} \\ 0 & X & Y & Z & 0 \\ 0 & X & R^n & P & 0 \\ 0 & 0 & C & C & 0 \\ \mbox{} & 0 & 0 & 0 & \mbox{}\\ };
         \path[overlay,->, font=\scriptsize,>=latex]
        (m-1-2) edge (m-2-2)
        (m-2-1) edge (m-2-2)
        (m-3-1) edge (m-3-2)
        (m-4-1) edge (m-4-2)
        (m-4-2) edge (m-5-2)
        (m-1-3) edge (m-2-3)
        (m-1-4) edge (m-2-4)
        (m-4-3) edge (m-5-3)
        (m-4-4) edge (m-5-4)
        (m-2-4) edge (m-2-5)
        (m-3-4) edge (m-3-5)
        (m-4-4) edge (m-4-5)
        (m-4-2) edge (m-4-3)
        (m-3-2) edge (m-4-2)
        (m-2-2) edge node[above]{$i_1$} (m-2-3)
        (m-2-3) edge node[above]{$p_1$} (m-2-4)
        (m-3-2) edge node[above]{$i_2 \circ i_1$} (m-3-3)
        (m-3-3) edge (m-3-4)
        (m-2-3) edge node[right]{$i_2$} (m-3-3)
        (m-3-3) edge node[right]{$p_2$} (m-4-3)
        (m-2-4) edge (m-3-4)
        (m-3-4) edge (m-4-4);
        \path[overlay,=,font=\scriptsize,>=latex]
        (m-2-2) edge[double,thick] (m-3-2)
        (m-4-3) edge[double,thick] (m-4-4);
        \end{tikzpicture}\]

Since $Z,C \in \Deep(R)$, it follows that $P \in \Deep(R)$ which shows that $X \in \Omega \Deep(R)$.

\item Since $X \in \DF(R)$, there is an exact sequence of the form $0 \rightarrow R \xrightarrow{i_2} X^n \xrightarrow{p_2} C \rightarrow 0$ for some $n$.  We also have the exact sequence $0 \rightarrow X^n \xrightarrow{i_1^n} Y^n \xrightarrow{p_1^n} Z^n \rightarrow 0$.  Letting $P$ be the pushout along $p_1^n$ and $i$, we have the following pushout diagram with exact rows and columns:
\[\begin{tikzpicture}[descr/.style={fill=white,inner sep=1.5pt}]
       \matrix (m) [
            matrix of math nodes,
            row sep=2em,
            column sep=2em,
            text height=1.5ex, text depth=0.25ex
        ]
        {\mbox{} & 0 & 0 & 0 & \mbox{} \\ 0 & R & X^n & C & 0 \\ 0 & R & Y^n & P & 0 \\ 0 & 0 & Z^n & Z^n & 0 \\ \mbox{} & 0 & 0 & 0 & \mbox{}\\ };
         \path[overlay,->, font=\scriptsize,>=latex]
        (m-1-2) edge (m-2-2)
        (m-2-1) edge (m-2-2)
        (m-3-1) edge (m-3-2)
        (m-4-1) edge (m-4-2)
        (m-4-2) edge (m-5-2)
        (m-1-3) edge (m-2-3)
        (m-1-4) edge (m-2-4)
        (m-4-3) edge (m-5-3)
        (m-4-4) edge (m-5-4)
        (m-2-4) edge (m-2-5)
        (m-3-4) edge (m-3-5)
        (m-4-4) edge (m-4-5)
        (m-4-2) edge (m-4-3)
        (m-3-2) edge (m-4-2)
        (m-2-2) edge node[above]{$i_2$} (m-2-3)
        (m-2-3) edge node[above]{$p_2$} (m-2-4)
        (m-3-2) edge node[above]{$i_1^n \circ i_2$} (m-3-3)
        (m-3-3) edge (m-3-4)
        (m-2-3) edge node[right]{$i_1^n$} (m-3-3)
        (m-3-3) edge node[right]{$p_1^n$} (m-4-3)
        (m-2-4) edge (m-3-4)
        (m-3-4) edge (m-4-4);
        \path[overlay,=,font=\scriptsize,>=latex]
        (m-2-2) edge[double,thick] (m-3-2)
        (m-4-3) edge[double,thick] (m-4-4);
        \end{tikzpicture}\]

But then $P \in \Deep(R)$ since $C,Z^n \in \Deep(R)$, and thus $Y \in \DF(R)$, as desired.

\end{enumerate}

\end{proof}

\begin{lem}\label{deepinjsurj}
Assume that $\depth(R)=0$, $M \in \Omega \Deep(R)$, and $N \in \DF(R)$. If $M \twoheadrightarrow N$ (resp. $N \hookrightarrow M$), then $R \mid M$ (resp. $R\mid M$ and $R\mid N$).
\end{lem}

\begin{proof}
First, assume that $M \twoheadrightarrow N$. As $M\in \Omega \Deep(R)$, if $R \nmid M$ then $\Soc (R)M=0$ by Lemma \ref{minsyz}. 

Since $M \twoheadrightarrow N$, $\Soc(R)N=0$. But $N$ is a nonzero faithful module which is a contradiction. Hence $R \mid M$. 

Now assume that $N\hookrightarrow M$. Then $M,N$ are both faithful and in $\Omega \Deep(R)$, so they have a free summand by Lemma \ref{minsyz}.
\end{proof}

\begin{prop}\label{homcutdown}
 Suppose that $M,N \in \Deep(R)$ and $\Ext ^{1 \leqslant i \leqslant t-1}_{R}(M,N)=0$. Then for any general regular sequence $\underline{x}$ of length $n\leq t$ we have $\overline{\Hom_{R}(M,N)} \cong \Hom_{\overline{R}} (\overline{M},\overline{N}) $ when $n<t$, and $\overline{\Hom_{R}(M,N)} \hookrightarrow \Hom_{\overline{R}} (\overline{M},\overline{N}) $ when $n=t$. 

\end{prop}

\begin{proof}
This is standard argument. Apply $\Hom_R(-,N)$ to the short exact sequence
$$0 \to M \to M\to M/xM \to 0$$
and proceed by induction. 
\end{proof}

\begin{lem}\label{minsyzcutdown}
If $M \in \Omega_{\min}\Deep(R)$, then for a general regular sequence $\underline{x} \subseteq \m$, $M/\underline{x}M \in \Omega_{\min} \Deep(R/\underline{x}R)$.
\end{lem}

\begin{proof}

Since $M$ is a minimal syzygy, there is an exact sequence $0 \to M \to R^{\mu(X)} \to X \to 0$ with $X \in \Deep(R)$.  Then any regular sequence $\underline{x}$ on $X$ suffices.

\end{proof}

\begin{lem}\label{summandcutdown}

If $M \in \Omega \Deep(R)$, then $R/\underline{x}R \mid M/\underline{x}M$ if and only if $R \mid M$, for a general regular sequence $\underline{x}$.
\end{lem}

\begin{proof}

One direction is clear.  If $\overline{R} \mid \overline{M}$, then $\overline{M}$ is not a minimal syzygy in $\Omega \Deep(\overline{R})$. Therefore $M$ is not a minimal syzygy in $\Omega \Deep(R) $ which implies $R \mid M$.

\end{proof}

\begin{lem}\label{intersectfreesummand}
If $M\in \Omega \Deep(R) \cap \DF(R)$ then $R \mid M$.
\end{lem}

\begin{proof}
For $t=0$ the result follows from Proposition \ref{minsyz}. By Lemma \ref{deepcutdown} we may find a regular sequence $\underline{x}$ so that $M/\underline{x}M \in \Omega \Deep(R/\underline{x}R) \cap \DF(R/\underline{x}R)$. So $R/\underline{x}R \mid M/\underline{x}M$ which implies $R \mid M$ by Lemma \ref{summandcutdown}.
\end{proof}

\begin{theorem}\label{corintersectfreesummand}
Suppose that $M \in \Omega \Deep(R)$ and $N \in \DF(R)$.
\begin{enumerate}
\item If $M \twoheadrightarrow N$ then $R \mid M$.
\item  If there is an exact sequence $0 \rightarrow N \rightarrow M \rightarrow X \rightarrow 0$ such that $X \in \Deep(R)$ then $R\mid M$.
\end{enumerate}

\end{theorem}

\begin{proof}
We cut down using a general regular sequence and appeal to \ref{summandcutdown} and \ref{intersectfreesummand}.
\end{proof}

\section{When does $\Hom_R(M,N)$ contain a free summand?}\label{MNfree}

For this section we retain the notation of Section \ref{deep}.

\begin{lem}\label{summandlem} 
Let $M,N$ be such that and $\Ext^{1 \le i \le t-1}(M,N)=0$.
\begin{enumerate}
\item Suppose $N \in \Omega \Deep(R)$. Then $\Hom_R(M,N) \in \Omega \Deep(R)$.  
\item Suppose $\Hom(M,N)\in \DF(R)$. Then $N \in \DF(R)$. 
\end{enumerate}
\begin{proof}

Consider part (1). Let 
\[\cdots  \rightarrow F_{1} \rightarrow F_{0} \rightarrow M\rightarrow 0\]
be a minimal free resolution of $M$. Since $\Ext^{1 \leqslant i \leqslant t-1}_R(M,N)=0$ we have the following exact sequence, given by applying $\Hom(-,N)$ to the resolution:

\[0\rightarrow \Hom_{R}(M,N) \rightarrow N^{l_0}\rightarrow N^{l_1} \rightarrow \cdots \rightarrow N^{l_{t}}\rightarrow C \rightarrow 0\]
Split this sequence into 
\[0\rightarrow \Hom_{R}(M,N) \rightarrow N^{l_0} \rightarrow X \rightarrow 0\]
and 
\[0\rightarrow X \rightarrow N^{l_1} \rightarrow \cdots \rightarrow N^{l_{t}} \rightarrow C \rightarrow 0\]
Since $N \in \Omega \Deep(R)$ it follows from the latter sequence that $X \in \Deep(R)$, and so $\Hom_R(M,N) \in \Omega \Deep(R)$, applying Lemma \ref{pushouts} to the former sequence.

Part (2) is proved similarly, also using Lemma \ref{pushouts}.
\end{proof}

\end{lem}

\begin{lem}\label{dualfreeisfree}

Suppose $t=0$ and $R \mid M^*$.  Then $R \mid M$. 

\begin{proof}

Take a minimal presentation $R^m \xrightarrow{A} R^n \rightarrow M \rightarrow 0$ of $M$.  This sequence induces an exact sequence of the form 

\[0 \rightarrow M^* \rightarrow R^n \xrightarrow{A^T} R^m \rightarrow \Tr M \rightarrow 0,\]
where $\Tr M$ denotes the Auslander transpose of $M$.  Setting $l=\mu(M^*)$, we have another exact sequence
\[R^l \xrightarrow{B} R^n \xrightarrow{A^T} R^m \rightarrow \Tr M \rightarrow 0\]
Since $R \mid M^*$, by Lemma \ref{minsyz}, $M^*$ is not a minimal syzygy, and so $B$ must contain a unit.  Thus there are invertible matrices $P$ and $Q$ so that $QBP^{-1}$ has the block form $\begin{pmatrix} 1 & 0 \\ 0 & B' \end{pmatrix}$.
But this gives rise to a chain isomorphism:

\[\begin{tikzpicture}[descr/.style={fill=white,inner sep=1.5pt}]
       \matrix (m) [
            matrix of math nodes,
            row sep=2.5em,
            column sep=4em,
            text height=1.5ex, text depth=0.25ex
        ]
        {R^l & R^n & R^m & \Tr M & 0 \\ R^l & R^n & R^m & \Tr M & 0 \\ };
         \path[overlay,->, font=\scriptsize,>=latex]
        (m-1-1) edge node[above]{$B$} (m-1-2)
        (m-1-2) edge node[above]{$A^T$} (m-1-3)
        (m-1-3) edge (m-1-4)
        (m-1-4) edge (m-1-5)
        (m-2-1) edge node[above]{$QBP^{-1}$} (m-2-2)
        (m-2-2) edge node[above]{$A^TQ^{-1}$} (m-2-3)
        (m-2-3) edge (m-2-4)
        (m-2-4) edge (m-2-5)
        (m-1-1) edge node[right]{$P$} (m-2-1)
        (m-1-2) edge node[right]{$Q$} (m-2-2);
        \path[overlay,-,font=\scriptsize,>=latex]
        (m-1-3) edge[double,thick] (m-2-3)
        (m-1-4) edge[double,thick] (m-2-4);
         \end{tikzpicture}\]

Since $QBP^{-1}$ has the form $\begin{pmatrix} 1 & 0 \\ 0 & B' \end{pmatrix}$ and since $A^TQ^{-1} \cdot QBP^{-1}=0$, it must be that $A^TQ^{-1}$ has a column of all $0$'s.  Hence $(Q^{T})^{-1}A$ has a row of all $0$'s.  But this implies that $R \mid \coker((Q^T)^{-1}A) \cong \coker(A)=M$, as desired.

\end{proof}

\end{lem}

\begin{lem}\label{dualizefree}

Suppose $\Ext^{1 \le i \le t}_R(M,R)=0$.  Then $M^*$ free implies $M$ is free.

\begin{proof}

We may suppose $M$ is not free.  Then there is part of a free resolution of $M$ of the form $F_{t+1} \to F_{t} \to \cdots \to F_1 \xrightarrow{A} F_0 \to M \to 0$ where $\im A \subseteq \m F_0$.  Dualizing this sequence, we obtain, since $\Ext^{1 \le i \le t}(M,R)=0$, an exact sequence
\[0 \to M^* \to F_0^* \xrightarrow{A^T} F_1^* \to \cdots \to F_{t}^* \to F_{t+1}^* \to C \to 0.\]
Split this sequence into exact sequences 
\[0 \to M^* \to F_0^* \to Y \to 0\]
and 
\[0 \to Y \to F_1^* \to \cdots \to F_{t}^* \to F_{t+1}^* \to C \to 0.\]
Then $Y$ is a $t+1$ syzygy of $C$.  But since $M^*$ is free, $\pd C \le t$ (recall that $t=\depth R$).  Hence $Y$ is a free summand of $F_1^*$.  By Lemma 1.4.7 in \cite{BH93} there is a unit in $A^T$ since $Y=\im A^T$.  But, by construction, $A$ has no unit, and we have a contradiction.

\end{proof}

\end{lem}

\begin{remark}

If $M \in \Deep(R)$, one can also derive Lemma \ref{dualizefree} from Lemma \ref{dualfreeisfree} by cutting down a regular sequence and appealing to the trace map $M\otimes M^* \to R$, as in \cite{VA68}.

\end{remark}

\begin{example}
Suppose $t>0$ and consider $M=R \oplus k$.  Then $M^* \cong R$, and $\Ext^{1 \le i \le t-1}(M,R)=0$, but of course $M$ is not free.  Thus, in general, one cannot reduce the $\Ext$ vanishing hypothesis to $\Ext^{1 \le i \le t-1}_R(M,R)=0$.

\end{example}

\begin{theorem}\label{freesummandsMN}
Let $M \in \Deep(R)$ and $N\in \Omega \Deep(R)$. Assume that $ \Hom _R(M,N) \in \DF(R)$ and  $\Ext^{1 \leqslant i \leqslant t-1}_R(M,N)=0$. Then $R \mid N$. 
\end{theorem}

\begin{proof}
By Lemma \ref{summandlem}, we have $\Hom(M,N) \in \Omega \Deep(R) \cap \DF(R)$. Hence $R \mid \Hom(M,N)$ by \ref{intersectfreesummand}. As $\overline{\Hom_{R}(M,N)} \hookrightarrow \Hom_{\overline{R}} (\overline{M},\overline{N})$ and $\overline{\Hom_{R}(M,N)} \in \DF(\overline R) $, $ \Hom_{\overline{R}} (\overline{M},\overline{N}) \in \DF(\overline{R})$. Therefore $\overline{N} \in \DF(\overline{R})$. Hence $\overline{R} \mid \overline{N}$ and so $R \mid N$ by from Lemmas \ref{minsyz}  and \ref{minsyzcutdown}.



\end{proof}

\begin{cor}\label{MNFree}
Let  $M \in \Deep(R)$ and $N \in \Omega \Deep(R)$. Furthermore, assume that $\Hom_{R}(M,N)$ is free and $\Ext^ {1\leqslant i \leqslant t-1}_R(M,N) = 0$. Then $N$ is free. 
\end{cor}

\begin{proof}

By Theorem \ref{freesummandsMN}, we have that $R \mid N$. Let $N=R^n \oplus N'$ where $n>1$.  Suppose $N' \ne 0$.  Since $N \in \Omega \Deep(R)$, $N' \in \Omega \Deep(R)$ by Lemma \ref{pushouts}.  Further, since $\Hom_R(M,N) \cong \Hom(M,R^n) \oplus \Hom(M,N')$ is free, $\Hom(M,N')$ is free and since $0=\Ext^{1 \le i \le t-1}_R(M,N) \cong \Ext^{1 \le i \le t-1}_R(M,R^n) \oplus \Ext^{1 \le i \le t-1}_R(M,N')$ we have that $\Ext^{1 \le i \le t-1}_R(M,N')=0$. Thus we may apply Theorem \ref{freesummandsMN} to obtain that $R \mid N'$.  Induction on $\mu(N)$ now shows that $N$ is free.

\end{proof}

\begin{theorem}\label{Extt}

Suppose $M \in \Deep(R)$, $N \in \Omega \Deep(R)$, $\Hom(M,N)$ is free, and $\Ext^{1 \le i \le t}(M,R)=0$.  Then $M$ is free.

\begin{proof}
By Corollary \ref{MNFree} we have that $N$ is free.  Thus $N \cong R^n$ for some $n$.  Now we have $\Hom(M,N) \cong (M^*)^n$ is free, and thus $M^*$ is free.  Thus $M$ is free by Lemma \ref{dualizefree}.

\end{proof}

\end{theorem}

Next we address the question of when $M^*$ is free. The most interesting case is when $\dim R\leq 1$.

\begin{lem}\label{depthlessthan1case}

Suppose $R$ is Cohen-Macaulay with dimension $d \le 1$, and suppose $M \in \CM(R)$.  Then $M^*$ free implies $M$ is free. 
\end{lem}

\begin{proof}

We have already obtained the result when $d=0$ by Lemma \ref{Extt}.  So we may suppose $d=1$.  Let $x$ be a general regular element on $R$ and $M$. Now, since $M^* \cong R^r$ for some $r$, we have $M_{\p} \cong R^r_{\p}$ for all $\p \in \Min R$, from Lemma \ref{Extt}.  Thus $M$ has constant rank $r$. In particular, $M_{\p} \cong M^*_{\p}$ for all $\p \in \Min R$.   We have (\cite[Theorem 4.6.8]{BH93}) 

\[e(x,M)=\sum_{p \in \Min R} e(x,R/\p)l(M_{\p})=\sum_{p \in \Min R} e(x,R/\p)l(M^*_{\p})=e(x,M^*).\]

Note that we have an exact sequence $0 \to \overline{M^*} \xrightarrow{i} \Hom(\overline{M},\overline{R})$
by Proposition \ref{homcutdown}. 

Set $n=\mu(\overline{M})$.  Then we have a map $\overline{R}^n \twoheadrightarrow \overline{M}$. Dualizing this gives a map $j:\Hom(\overline{M},\overline{R}) \hookrightarrow \overline{R}^n$.  Since $\overline{M^*} \cong \overline{R}^r$, this gives us an exact sequence of the form
\[0 \rightarrow \overline{R}^r \xrightarrow{j \circ i} \overline{R}^n \rightarrow C \rightarrow 0.\]
But then $\pd C<\infty$ which means $C$ is free, since $\depth \overline{R}=0$.  Thus this sequence splits, whence the map $i$ is a split injection.  Thus $\Hom(\overline{M},\overline{R}) \cong \overline{R}^r \oplus L$ for some $L$.  By Lemma \ref{dualfreeisfree} applied repeatedly (as Krull-Schmidt holds over $\bar R$ because it is Artinian), it follows that $\overline{R}^r \mid \overline{M}$.  But since $l(\overline{M})=e(x,M)=e(x,M^*)=l(\overline{M^*})=l(\overline{R}^r)$, it must be that $\overline{M} \cong \overline{R}^r$ which implies $M$ is free.

\end{proof}

Standard arguments now allow us to show that the freeness of $M^*$ forces that of $M$ in general.

\begin{theorem}\label{generalMdual}

Suppose $R$ and $M$ satisfy $(S_2)$.  Then $M^*$ free implies $M$ is free. 
\end{theorem}

\begin{proof}

It suffices to show that $M$ is reflexive.  We assume $d=\dim R\geq 2$ as the small dimension case was covered by \ref{depthlessthan1case}.  Also by  \ref{depthlessthan1case}, we have that $M_{\p}$ is free, in particular, reflexive, for all $\p \in \Spec R$ with $\hit \p \le 1$. The natural map $M\to M^{**}$ is an isomorphism in codimension one, so is an isomorphism (or one can appeal to  \cite[Proposition 1.4.1]{BH93}).

\end{proof}

In the next part we extend one of the main results of \cite{HL04}:

\begin{theorem} (Huneke-Leuschke)\label{HunekeLeuschke}

Suppose $R$ is Cohen-Macaulay and is a complete intersection in codimension $1$. Furthermore, assume that $\Q \subseteq R$.  If $M$ is an $R$-module that is locally free in codimension one with constant rank, $\Ext^{1 \le i \le d}_R(M,M)=0$, and $\Ext^{1 \le i \le 2d+1}_R(M,R)=0$, then $M$ is free.

We start with a very well-known fact about shifting $\Ext$ modules.

\begin{lem}\label{Ext}

If $\Ext^i_R(M,R)=\Ext^i_R(\Omega M,\Omega N)=0$ then $\Ext^i_R(M,N)=0$ and if $\Ext^{i+1}_R(M,R)=\Ext^i_R(M,N)=0$, then $\Ext^i_R(\Omega M,\Omega N)=0$

\begin{proof}

The exact sequence $0 \to \Omega N \to R^n \to N \to 0$ induces an exact sequence 
\[\Ext^i_R(M,R^n) \to \Ext_R^i(M,N) \to \Ext_R^{i+1}(M,\Omega N) \to \Ext_R^{i+1}(M,R)\]
from which we obtain the result.  

\end{proof}

\end{lem}

\end{theorem}

\begin{lem}\label{lemHunekeLeuschkegen}

Suppose $R$ is $(S_2)$, $(G_1)$, and equidimensonal.  Further suppose $\Q \subseteq R$ and that $R$ admits a canonical module $\omega_R$, in the notion of \cite{HH94}.
Let $M \in \Omega \Deep(R)$ be a reflexive $R$-module, free in codimension $1$, and suppose $\Ext^{1 \le i \le t-1}_R(M,M)=0$.  Then $M$ is free.  

\begin{proof}

We may suppose $\dim R \ge 2$ since $M$ is free in codimension $1$.  We claim that $M$ has constant rank.  Take $\p,\q \in \Min R$.  Since $R$ satisfies $(S_2)$, the Hochster-Huneke graph of $R$ is connected (see \cite{HH94}).  This means there is a chain of minimal primes $\p=\p_1,\p_2,\dots,\p_n=\q$ such that $\hit (\p_i+\p_{i+1}) \le 1$.  Ergo, $\rank M_{\p_i}=\rank M_{\p_{i+1}}$ for each $i$, since $M$ is free on a minimal prime of $\p_i+\p_{i+1}$.  In particular, $\rank M_{\p}=\rank M_{\q}$, and so $M$ has constant rank.

Now, we have $R \mid \End_R(M)$ from the trace map as explained in \cite[Appendix]{HL04}. By Theorem \ref{freesummandsMN}, $M=R \oplus M'$.  But now $M'$ satisfies the hypotheses again by Lemma \ref{pushouts}, and so, proceeding inductively on the number of generators gives that $M$ is free.  

\end{proof}

\end{lem}

\begin{theorem}\label{HunekeLeuschkegen}

Suppose $R$ is $(S_2)$, $(G_1)$, equidimensional and suppose $\Q \subseteq R$.  Suppose $N \in \mod(R)$ such that $\pd N_{\p}<\infty$ for all $\p \in \Spec R$ with $\hit \p \le 1$.  Set $a=\min\{t,\depth N\}$ and suppose $\Ext^{1 \le i \le t-1}_R(N,N)=\Ext^{1 \le i \le 2t+1-a}_R(N,R)=0$.  Then $N$ is free.

\begin{proof}

Set $M=\Omega^{t+2-a}(N)$.  Then $M \in \Omega \Deep(R)$ and $M$ is reflexive.  This gives us that $R \mid \End(M)$.  By Lemma \ref{Ext}, we have $\Ext^{1 \le i \le t-1}(M,M)=0$.  By Lemma \ref{lemHunekeLeuschkegen}, $M$ is free.  Thus $\pd N \le t+2-a$.  If $\pd N=l$, then $\Ext^l_R(N,X) \ne 0$ for every finitely generated $X \ne 0$.  But $t+2-a \le 2t+1-a$ and so it must be that $l=0$. Therefore, $N$ is free.

\end{proof}

\end{theorem}

Next we discuss and extend a result by Goto-Takahashi (\cite[Corollary 4.3]{GT17}). First, we recall their result and give a somewhat simpler proof. Note that their result does not follow directly from our previous results since, for instance, $I$ may not be in $\Omega\Deep(R)$.  

\begin{theorem}(Goto-Takahashi)\label{idealfree}

Suppose $R$ is CM and that $I$ is a CM ideal of height $1$.  Assume that
\begin{enumerate}
\item $\Hom_R(I,I)$ is free. 
\item $\Ext^{1 \le i \le d}_R(I,R)=0$. 
\item $\Ext^{1 \le i \le d-1}_R(I,I)=0$.
\end{enumerate}
Then $I$ is free.

\end{theorem}

\begin{proof}

By standard reduction arguments, see \cite[Theorem 3.3]{GT17} one can assume $\dim R=1$.  By prime avoidance, there exists $a \in I$ which is not in $\Min(R) \cup {\m}I$.  Thus $a$ is part of a minimal generating set for $I$, and we have an exact sequence of the form 
\[0 \rightarrow R \rightarrow I \xrightarrow{f} C \rightarrow 0.\]

Since $I$ has height $1$, it follows that $R/I$ has finite length.  Thus, localizing the exact sequence $0 \to I \to R \to R/I \to 0$ at any $\p \in \Min(R)$, we obtain that $I$ has constant rank $1$.  Ergo, $C$ has finite length.  In general, if $g \in \m\Hom(X,Y)$ then $\im g \in {\m}Y$.  Thus the above argument gives us that $f$ is part of a minimal generating set for $\Hom(I,C)$.  Now, since $\Hom(I,I)$ is free, it must be that $\Hom(I,I) \cong R$, since $I$ has rank $1$.  Since $\Ext^1(I,R)=0$, we have the exact sequence 

\[0 \rightarrow \Hom(I,R) \rightarrow R \rightarrow \Hom(I,C) \rightarrow 0\]
Thus $\Hom(I,C)$ is cyclic, and $\{f\}$ a generating set.  By construction, $f(a)=0$, and thus for every $g \in \Hom(I,C)$, we have $g(a)=0$.  But on the other hand, $\Hom(I,\Soc C) \hookrightarrow \Hom(I,C)$ and the former is isomorphic to $\Hom(I/{\m}I,\Soc(C))$.  But as this is a vector space, if $C \ne 0$, we may find a map $h \in \Hom(I,C)$ so that $h(a) \in \Soc(C)-\{0\}$ and $h(x)=0$ for any minimal generator $x \ne a$.  But this is a contradiction, and so $C=0$, whence $I \cong R$.

\end{proof}

The next theorem extends the Goto-Takahashi result to modules of higher rank. 

\begin{theorem}\label{GTgen}
Let $R$ be a Cohen-Macaulay local normal domain. Let $M$ be a maximal Cohen-Macaulay module such that $\Hom_R(M,M)$ is free and $$\Ext^{1\leq i\leq d-1}_R(M,M)= \Ext^{1\leq i\leq d}_R(M,R)=0.$$ Then $M$ is free. 
\end{theorem}

\begin{proof}
We employ a standard argument in the theory of Brauer groups. Let $S= R^{sh}$ denote the strict Henselization of $R$. Then $S$ is still a local normal domain, and it is harmless to replace  $R$ by $S$ without affecting the assumptions and desired conclusion. Thus we assume $R$ is Henselian with a separably closed residue field $k$. Let $A=\End_R(M)$ and set $r=\rank_R M$. Then as $M$ is reflexive and $A$ is a free module of rank $r^2$, $A$ is an Azumaya algebra (see for example \cite{CG75}, proof of Corollary 1.4). Then so is the $k$-algebra $B=A\otimes_R k$. Since $k$ is separably closed, $B$ is actually isomorphic as an algebra to $\End_k(k^r)$. Now as $R$ is Hensenlian, one can lift idempotents, which shows that $M$ splits into a direct sum of ideals. These ideals inherit all the assumptions, so by Theorem \ref{idealfree} they are all free, and so is $M$.    
\end{proof}

\begin{remark}
If $R^{sh}$ is an UFD, then our argument shows that $M$ is free without any assumption on vanishing of $\Ext$ modules. 
\end{remark}

\section{When is $\Hom(M,N) \cong N^{r}$?}\label{MNN}

In this section we try to understand the question in the title. Let $t$ be some fixed integer. Unlike the previous sections, we don't necessarily assume $\depth R=t$.

Set $\upsilon_i(M) = \dim_{k} \Ext ^{i}_R(k, M)$.  We let $I_j(M)$ denote the $j$-th fitting ideal ideal of $M$, namely the ideal generated by $(n-j)$-minors of any presentation matrix $A$ of $M$ in a sequence:$$R^m \xrightarrow{A} R^n \rightarrow M \rightarrow 0$$

We first recall a result (\cite[Lemma 2.1]{GT17}). For completeness, we provide an elementary proof that avoids spectral sequences. 

\begin{prop}(Goto-Takahashi)\label{type}
Let $M, N$ be such that $\depth(M),\depth(N)\geq t$. Assume that $\Ext^{1 \leqslant i \leqslant t} _R (M, N)=0$. Then $\upsilon_t(\Hom(M, N))= \mu(M)\upsilon_t(N)$.
\end{prop}

\begin{proof}

Take $F_t \to \cdots \to F_1 \to F_0 \to M \to 0$ to be part of a (possibly non-minimal) free resolution of $M$ where $F_i \ne 0$ for each $i$.  Note that such a resolution exists even if $\pd M<t$.  Then a similar argument to that of Lemma \ref{summandlem} shows that $\depth \Hom(M,N) \ge t$.
Now, we have $\overline {\Hom_R(M, N)} \cong \Hom _{\overline{R}}(\overline{M}, \overline{N})$ for a general regular sequence of length $t$. Hence $\upsilon_t(\Hom_R(M, N))=\dim_k \Hom(k,\overline {\Hom_R(M, N)})  = \dim_{k} \Hom (k, \Hom(\overline{M}, \overline{N})) = \dim_{k} \Hom(k\otimes \overline{M}, \overline{N}) =  \mu(M) \cdot \upsilon_t(N)$.

\end{proof}

\begin{theorem}\label{fitting}

Suppose $\dim N=0$.  Then $\Hom(M,N) \cong N^r$ if and only if $\mu(M)=r$ and $I_{r-1}(M)N=0$.  

\begin{proof}
First we consider the case were $\dim R=0$

[$\Rightarrow$] Since $\Hom(M,N) \cong N^r$ we have, by Proposition \ref{type} $\upsilon_0(\Hom(M,N))=\upsilon_0(N^r)$ and so $\mu(M) \upsilon_0(N)=r\upsilon_0(N)$ from which we see that $\mu(M)=r$ as $\upsilon_0(N)\neq 0$.  We also have $l(M \otimes N^{\vee})=l((N^r)^{\vee})=rl(N^{\vee})$ by Matlis duality.  Take a minimal presentation 

\[R^m \xrightarrow{A} R^r \rightarrow M \rightarrow 0.\]

Tensoring  with $N^{\vee}$ we have the exact sequence 
\[(N^{\vee})^m \xrightarrow{A \otimes id_{N^{\vee}}} (N^{\vee})^r \rightarrow M \otimes N^{\vee} \rightarrow 0.\]  

But $l((N^{\vee})^r=l(M \otimes N^{\vee})$ since $((N^{\vee})^r)^{\vee} \cong N^r$ and $(M \otimes N^{\vee})^{\vee} \cong \Hom(M,N)$.  
Thus $\im(A \otimes \id_{N^{\vee}})=0$ which implies $I_{r-1}(M) \subseteq \Ann(N^{\vee})=\Ann(N)$.  

 [$\Leftarrow$] Suppose $\mu(M)=r$ and $I_{r-1}(M)N=0$.  Since $\mu(M)=r$, we may take a minimal presentation of $M$ of the form 
\[R^m \xrightarrow{A} R^r \rightarrow M \rightarrow 0.\]

Tensoring  with $N^{\vee}$ we have the exact sequence 
\[(N^{\vee})^m \xrightarrow{A \otimes id_{N^{\vee}}} (N^{\vee})^r \rightarrow M \otimes N^{\vee} \rightarrow 0.\]  
Since $I_{r-1}(M) \subseteq \Ann(N)=\Ann(N^{\vee})$, we have $\im(A \otimes \id_{N^{\vee}})=0$ and thus $(N^{\vee})^r \cong M \otimes N^{\vee}$. That $\Hom(M,N) \cong N^r$ now follows from Matlis duality.  So we have the result when $\dim R=0$.

Now suppose $\dim R>0$. Set $I=\Ann N$.  Since $N$ has finite length, $I$ is $\m$-primary.  Set $\overline{(-)}=(-) \otimes R/I$.  Then $\Hom(M,N) \cong \Hom_{\overline{R}}(\overline{M},N)$ so that $\Hom_{\overline{R}}(\overline{M},N) \cong N^r$.  But from the dimension $0$ result, this holds if and only if $I_{r-1}(\overline{M})=\overline{I_{r-1}(M)}N=0$ which precisely means that $I_{r-1}(M) \subseteq \Ann N$, as desired.

\end{proof}

\end{theorem}

\begin{cor}\label{fittingM}

If $\dim M=0$ and $\Hom(M,M) \cong M^r$ then $I_{r-1}(M)=\Ann M$. 

\begin{proof}

We always have $\Ann M \subseteq I_{r-1}(M)$.  The result follows from combining this fact with the Theorem \ref{fitting}.

\end{proof}

\end{cor}

\begin{cor} \label{fittingbest}
Suppose $\Hom(M,N) \cong N^r$ and suppose $\Ass N=\Min N$.  Then $I_{r-1}(M)N=0$.

\begin{proof}

For any $\p \in \Min(I_{r-1}(M)N)$, we have $\Hom_{R_{\p}}(M_{\p},N_{\p}) \cong (N_{\p})^r$.  

Since $(I_{r-1}(M)N)_{\p} \hookrightarrow N_{\p}$ it follows that $N_{\p}$ has depth $0$.  Thus $\p \in \Ass(N)=\Min(N)$ and so $N_{\p}$ has finite length.  Theorem \ref{fitting} gives us that $I_{r-1}(M_{\p})N_{\p}=(I_{r-1}(M))_{\p}N_{\p}=(I_{r-1}(M))N)_{\p}=0$.  But this says that $(I_{r-1}(M))N)_{\p}=0$ for all $\p \in \Min(I_{r-1}(M)N)$ which implies $I_{r-1}(M)N=0$.

\end{proof}

\end{cor}

\begin{remark}\label{fittingbestsharp}

We remark that the converse to Corollary \ref{fittingbest} does not hold.  Indeed, if $M$ has constant rank $r$, then $I_{r-1}(M)=0$, while $\Hom(M,N)$ need not be isomorphic to $N^r$.  To be more explicit, one could take $R$ with $\depth R=1$ and let $M=\m$ and $N=R$. 

\end{remark}

\begin{lem}\label{PM}
Let $M, N, P$ be nonzero $R-$modules such that $\Ass(P) \subseteq \Ass(N)=\Min(N)$. Suppose that $\Hom_R(M, N) \cong \Hom_R(P, N)$ and $\Ext_{R}^1(M, N)=0$. Then if $P\twoheadrightarrow M$, $P \cong M$. 
\end{lem}

\begin{proof}
First assume that $\dim(N)=0$. From the exact sequence 
$0\rightarrow X\rightarrow P\rightarrow M\rightarrow 0$ we have
$$ 0\rightarrow \Hom(M, N)\rightarrow \Hom(P, N)\rightarrow \Hom(X, N)\rightarrow 0.$$
By assumption $l(\Hom(X, N))= 0$. Therefore $X=0$.

If $\dim(N)> 0$, for any $\p \in \Min(N) = \Ass(N)$, $X_{\p} = 0$. So $\Ass(\Hom(X,N))= \Supp(X) \cap \Ass(N) = \emptyset$, and thus $\Hom(X,N)=0$. 

Now, if $X\neq 0$,  take $\q \in  \Min(X)$. Then $\q \in \Ass(P) \subseteq \Ass(N)$. But then $\Hom(X_{\q},N_{\q}) \neq 0$, a contradiction.   
\end{proof}

\begin{theorem}\label{Mfree}
Assume that  $\depth(M) \geqslant t$, $\depth(N)=t$, $\Ass(N)=\Min(N)$, and for some $s\geq t$, $\Ext^{1 \leqslant i \leqslant s} _R (M, N)=0$. If $\Hom(M, N) \cong N^{r}$ for some $r \in \N$, then $M/IM \cong (R/I)^r$ for $I=\Ann(N)$.

Furthermore, if  one of the following holds:
\begin{enumerate}
\item $N$ is faithful. 
\item $\Ass(R) \subseteq \Ass(N)$ and $s>0$.  
\end{enumerate}  
then $M \cong R^{r}$.
\end{theorem}

\begin{proof}
By Proposition \ref{type}, $v_t(N^{r})=\mu (M) \cdot v_t(N)$. Hence $\mu (M)=r$.
Since $\Ass(N)=\Min(N)$, Corollary \ref{fittingbest} tells us that $I_{r-1}(M)\subseteq I$. Since $M/IM$ is still $r$-generated over $R/I$, it must be a free $R/I$ module of rank $r$. 

For the furthermore statements, if $I=0$ then $M \cong R^r$. 
Assume the second set of conditions.  By Lemma \ref{PM}, we have $M \cong R^{r}$.
\end{proof}

\begin{cor}\label{extension}
Let $R \rightarrow S$ be a finite local homomorphism of local rings. Assume that $S$ is regular of dimension $t$, $\depth M \ge t$, and $\Ext^{1 \le i \le t}_R(M,S)=0$.  If one of the following holds 

\begin{enumerate}

\item $S$ is faithful as an $R$-module.

\item $\Ass_R R \subseteq \Ass_R S$.

\end{enumerate}

Then $M$ is free.

\end{cor}

\begin{proof}
Since $\Ext^{1 \le i \le t}_R(M,S)=0$, $\Hom_R(M,S) \in \CM(S)$, but since $S$ is regular we have $\Hom_R(M, S) \cong S^{l}$ for some $l \in \N$. Hence $M$ is free, by Theorem \ref{Mfree}.
\end{proof}

The following example shows that the conditions of \ref{Mfree} and \ref{extension} are needed. 

\begin{example}\label{conditionsneeded}

Let $R=k \llbracket x,y \rrbracket/(xy)$, let $S=R/(x)$ and $M=R/(x)$ as in Corollary \ref{extension}.  Then $\Ext^1_R(M,S)=0$ but of course $M$ is not free.

\end{example}

It is worth noting that our results in this section can also be viewed as modest confirmation of the Auslander-Reiten conjecture. For example Theorem \ref{Mfree} gives:

\begin{cor}\label{MM}
Let $M = R/I$ and $\depth M=t$. Assume that $\Ass(R)\subseteq \Ass(M)=\Min(M)$. Then $M$ is free if $\Ext^{1\leq i\leq \max\{1,t\}}_R(M,M)=0$.  
\end{cor}

\begin{proof}
Obviously $\Hom_R(M,M)\cong M$, so we can apply Theorem \ref{Mfree}.
\end{proof}

\section{Some other applications}
In this section we treat some similar problems that have appeared in the literature. The first one involves tests for Gorensteiness, in the spirit of \cite{UL84}.  Throughout this section we assume $R$ is a Cohen-Macaulay local ring with $\dim R=d$ and with canonical module $\w$.

\begin{cor}\label{testgor}
Suppose $\Ext^{1 \le i \le d}_R (M,R)=0$ and $M$ is Cohen-Macaulay in codimension $1$. If $M^{\vee} \cong M^* $ then $R$ is Gorenstein.
\end{cor}

\begin{proof}
Since $M$ is Cohen-Macaulay in codimension $1$, the natural map $M \to M^{\vee \vee}$ is an isomorphism in codimension $1$, thus an isomorphism.  Since $\Ext^{1 \le i \le d}(M,R)=0$ it follows, as in the proof of Lemma \ref{summandlem} that $M^* \in \CM(R)$ and so $(M^*)^{\vee} \cong M^{\vee \vee} \cong M \in \CM(R)$.  

By assumption and Proposition \ref{type} we have $$ \upsilon_d({M^\vee})= \upsilon_d(M^*) = \mu(M)\upsilon_d(R).$$
Since $\upsilon_d(M^\vee) = \mu(M)$ (one can appeal to Proposition \ref{type} again), $R$ has type one, and so is Gorenstein.
\end{proof}

\begin{remark}
The above was inspired by Theorem 2.1 of \cite{UL84}. The situation there is as follows. Let $R \rightarrow S$ be a finite extension with $\dim S=\dim R$ and $S$ is Cohen-Macaulay, local, and factorial. Under mild conditions, $\Hom_R(S,R)$ is isomorphic as an $S$ module to a rank one reflexive ideal of $S$, thus $\Hom_R(S,R)\cong S$. Also $\Hom_R(S,\w_R)\cong \w_S\cong S$. One can now appeal to \ref{testgor}, with $M=S$ to give an $\Ext$-vanishing test for the Gorensteiness of $R$.
\end{remark}

For completeness, we give the following, which extends \cite[Lemma 2.1]{HH05}, \cite[Theorem 2.7]{JO09}, and  \cite[Theorem 5.9]{HJ03}.

\begin{lem}\label{genhunekehanes}
Let $M,N \in \CM(R)$  Consider the conditions:

\begin{enumerate}

\item $\Ext_R^{1 \le i \le d}(M,N)=0$. 

\item $M \otimes N^{\vee}$ is in $\CM(R)$.

\end{enumerate}

Then $(1) \Rightarrow (2)$. If $\Ext_R^{1 \le i \le d}(M,N)$ have finite length then $(2) \Rightarrow (1)$.

\end{lem}

\begin{proof}

Suppose $\underline{x}$ is a regular sequence.  Set $e(M):=e(\underline{x},M)$.  First
we have $e(\Hom(M,N))=e(M \otimes N^{\vee})$ by the Associativity formula (\cite[Theorem 4.6.8]{BH93}).  Then, since $\Hom(M,N) \in \CM(R)$, we have 
\[e(\Hom(M,N))=l(\overline{\Hom(M,N)})=l(\Hom(\overline{M},\overline{N}))=l(\overline{M} \otimes \overline{N}^{\vee}).\]
But this says $e(M \otimes N^{\vee})=l(\overline{M} \otimes \overline{N}^{\vee})=l(\overline{M \otimes N^{\vee}})$ from which we deduce that $M \otimes N^{\vee}$ is MCM.


For the converse, first consider the case where where $\dim R=1$. Let $x$ be $R$-regular and $x\Ext^1_R(M,N)=0$.  The short exact sequence $0 \rightarrow M \xrightarrow{x} M \rightarrow M/xM \rightarrow 0$ induces the exact sequence 
\[0 \rightarrow \Hom(M,N) \xrightarrow{x} \Hom(M,N) \rightarrow \Hom(M/xM,N/xN) \rightarrow \Ext^1(M,N) \rightarrow 0.\]
In this case, it suffices then, to show that $\Hom(M,N)/x\Hom(M,N) \cong \Hom(M/xM,N/xN)$.  Since $M \otimes N^{\vee}$ is MCM, we see that $\dfrac{(M \otimes N^{\vee})^{\vee}}{x(M \otimes N^{\vee})^{\vee}} \cong \Hom(\dfrac{M \otimes N^{\vee}}{x(M \otimes N^{\vee})},\dfrac{\w}{x\w})$. But since $(M\otimes N^{\vee})^{\vee} \cong \Hom(M,N)$, this gives us that $\Hom(M,N)/x\Hom(M,N) \cong \Hom(M/xM,N/xN)$ so that $\Ext^1(M,N)=0$.  

Now suppose $\dim R>1$ and choose a regular $x \in \bigcap_{1 \le i \le d} \Ann(\Ext^i(M,N))$.  Then the long exact sequence in $\Ext$, coming from the short exact sequence $0 \rightarrow M \xrightarrow{x} M \rightarrow M/xM \rightarrow 0$, decomposes into short exact sequences 
\[0 \rightarrow \Ext^i(M,N) \to \Ext^{i+1}(M/xM,N) \to \Ext^{i+1}(M,N) \to 0\]
for each $1 \le i \le d-1$.  Further, we have $\Ext^{i+1}(M/xM,N) \cong \Ext^{i}_{R/xR}(M/xM,N/xN)$ for each $1 \le i \le d-1$.  Thus it suffices to show $\Ext^{i}_{R/xR}(M/xM,N/xN)=0$ for each $1 \le i \le d-1$.  But since $M \otimes N^{\vee}$ MCM implies $M/xM \otimes (N/xN)^{\vee}$ is MCM over $R/xR$, this follows from induction, and we're done.

\end{proof}

To explore the previous Theorem a bit more, we make the following definition.  A pair of modules $M,N \in \CM(R)$ is called {\it tight} if $\Ext_R^{1\leq i \leq d}(M,N)=0$ forces $\Ext_R^i(M,N)=0$ for all $i>0$.  

\begin{remark}
A pair $(M,N)$ of maximal Cohen-Macaulay modules is tight in any of the following situations:
\begin{enumerate}

\item $M$ is free or $N$ has finite injective dimension. 

\item $M$ is locally free in codimension one, $M^*$ is maximal Cohen-Macaulay and $N= (M^{*})^{\vee}$ (when $R$ is Goresntein the last two conditions simply mean $N=M$). Here, the vanishing of $\Ext^{1\leq i\leq d}_R(M,N)$ forces $M\otimes_RM^*$ to be Cohen-Macaulay by Theorem \ref{genhunekehanes}, so the map $i:M\otimes M^* \to \Hom_R(M,M)$ would be an isomorphism as $i$ is already an isomorphism in codimension one. So $M$ is free.

\item $M$ has finite complete intersection dimension and the complexity of $M$ is at most $d-1$ \cite[Theorem 1.2]{CD11}.

\end{enumerate}

\end{remark}

\begin{cor}\label{tight}

Suppose $R$ is Cohen-Macaulay with canonical module $\omega$ and let $M,N\in \CM(R)$ such that for all $\p \in \Spec(R)-\{\m\}$, the pair $(M_{\p}, N_{\p})$ is tight. Then the following are equivalent:

\begin{enumerate}

\item $\Ext^{1 \le i \le d}_R(M,N)=0$. 

\item $M \otimes N^{\vee} \in \CM(R)$ .

\end{enumerate}

\end{cor}
\begin{proof}
Assume (2). Let $\p$ be a non-maximal prime. By induction on $\dim R$, $\Ext^{1\leq i\leq \hit \p}_{R}(M,N)_{\p}=0$. By assumption on tightness of the pair $(M_{\p}, N_{\p})$, $\Ext^{1\leq i\leq d}_{R}(M,N)_{\p}=0$.  So the modules $\Ext^{1\leq i\leq d}_R(M,N)$ have finite length, and we are done. 
\end{proof}

\section*{Acknowledgments}
It is a pleasure to dedicate this paper to Craig Huneke on the occasion of his 65th birthday. Our approach and many statements owe their form to a careful study of his work. We also thank Ryo Takahashi for many helpful discussions. The first author is partially supported by NSA grant FED0073853. The second author is supported by a grant from the Ministry of Science, Research and Technology, Iran.

\nocite{*}

\bibliographystyle{amsalpha}
\bibliography{mybib}

\end{document}